\documentclass[12pt]{article}
\usepackage{amssymb,amsmath,amsthm}
\usepackage{graphicx}
\usepackage{hyperref}

\usepackage{fullpage}

\newcommand\ex{\ensuremath{\mathrm{ex}}}

\usepackage{color}

\newcommand\cN{{\mathcal N}}

\theoremstyle{plain}
\newtheorem{theorem}{Theorem}
\newtheorem{lemma}[theorem]{Lemma}

\newtheorem{conjecture}[theorem]{Conjecture}
\newtheorem{proposition}[theorem]{Proposition}

\theoremstyle{definition}

\newtheorem{defn}[theorem]{Definition}

\newtheorem{cor}[theorem]{Corollary}

\newtheorem{prop}[theorem]{Proposition}

\newcommand\cref[1]{Corollary~\ref{cor:#1}}

\title{Some exact results for generalized Tur\'an problems}
\author{D\'aniel Gerbner\footnote{Alfr\'ed R\'enyi Institute of Mathematics, Hungarian Academy of Sciences, E-mail: \texttt{gerbner@renyi.hu.} Research supported by the J\'anos
    Bolyai Research Fellowship of the Hungarian Academy of Sciences and the
    National Research, Development and Innovation Office -- NKFIH under the
    grants FK 132060, KKP-133819, KH130371 and SNN 129364.}, \enskip Cory Palmer\footnote{Department of Mathematical Sciences,
University of Montana, Missoula, Montana 59812, USA. E-mail: \texttt{cory.palmer@umontana.edu.}}}

\begin{document}

\maketitle

\begin{abstract}
Fix a $k$-chromatic graph $F$. In this paper we consider the question to determine for which graphs $H$ does the Tur\'an graph $T_{k-1}(n)$ have the maximum number of copies of $H$ among all $n$-vertex $F$-free graphs (for $n$ large enough). We say that such a graph $H$ is {\it $F$-Tur\'an-good}. In addition to some general results, we give (among others) the following concrete results:
\begin{enumerate}
    \item[(i)] For every complete multipartite graph $H$, there is $k$ large enough such that $H$ is $K_k$-Tur\'an-good. 
    \item[(ii)] The path $P_3$ is $F$-Tur\'an-good for $F$ with $\chi(F) \geq 4$.
    
    \item[(iii)] The path $P_4$ and cycle $C_4$ are $C_5$-Tur\'an-good.
    
    \item[(iv)] The cycle $C_4$ is $F_2$-Tur\'an-good where $F_2$ is the graph of two triangles sharing exactly one vertex.
\end{enumerate}

\end{abstract}

\section{Introduction}
Fix a graph $F$. We say that a graph $G$ is {\it $F$-free} if it does not contain $F$ as a subgraph. 
A cornerstone of extremal graph theory is Tur\'an's theorem \cite{turan}, which determines the maximum number of edges in an $n$-vertex $K_k$-free graph. The extremal construction is a complete $(k-1)$-partite graph on $n$ vertices such that each class has cardinality either $\lceil n/(k-1)\rceil$ or $\lfloor n/(k-1)\rfloor$. Such a graph is called a {\it Tur\'an graph} and is denoted $T_{k-1}(n)$.

Tur\'an's theorem is the starting point of many avenues of research. The {\it Tur\'an function} $\ex(n,F)$ is the maximum number of edges in an $n$-vertex $F$-free graph. In this notation, Tur\'an's theorem states $\ex(n,K_k)=|E(T_{k-1}(n))|$. We call an $n$-vertex $F$-free graph with $\ex(n,F)$ edges an {\it extremal graph for $F$}. Thus, the Tur\'an graph $T_{k-1}(n)$ is the extremal graph for $K_k$.
The fundamental Erd\H os-Stone-Simonovits theorem \cite{es,ersi} states that if the chromatic number of $F$ is $k\ge 2$, then
\[
\ex(n,F)=(1+o(1))|E(T_{k-1}(n))|=\left(1-\frac{1}{k-1}+o(1)\right)\binom{n}{2}.
\]
Simonovits \cite{sim} characterized those graphs $F$ that have the Tur\'an graph as their unique extremal graph.
We say that an edge $e$ of a graph $F$ is \textit{color-critical} if deleting $e$ from $F$ results in a graph with smaller chromatic number.

\begin{theorem}[Color-critical edge theorem, Simonovits \cite{sim}]\label{crit-thm}
Let $F$ be a $k$-chromatic graph. For $n$ large enough, the unique extremal graph for $F$ is 
the Tur\'an graph $T_{k-1}(n)$ if and only if $F$ has a color-critical edge.
\end{theorem}

Here we consider a generalization of the results described above.
Fix graphs $H$ and $G$. Denote the number of copies of $H$ in $G$ by $\cN(H,G)$. Now fix graphs $F$ and $H$. Define
\[
\ex(n,H,F):=\max\{\cN(H,G): G \text{ is an $n$-vertex $F$-free graph}\},
\]
i.e., $\ex(n,H,F)$ is the maximum number of copies of the graph $H$ in an $n$-vertex $F$-free graph.
An early result of Zykov \cite{zykov} (see also Erd\H os \cite{Er1}) determines the maximum number of copies of $K_r$ in a $K_k$-free graph.

\begin{theorem}[Zykov \cite{zykov}]\label{zykov-theorem}
The Tur\'an graph $T_{k-1}(n)$ is the unique $n$-vertex $K_k$-free graph with the maximum number of copies of $K_r$. Thus,
\[
\ex(n,K_r,K_k) = \mathcal{N}(K_r, T_{k-1}(n)) \leq \binom{k-1}{r} \left \lceil \frac{n}{k-1}\right \rceil ^r.
\]
\end{theorem}
After several other sporadic results (see, e.g., \cite{BoGy, G2012, Gy, LiGy, gypl, HaETAL}), the general investigation of this function was initiated by Alon and Shikhelman \cite{alonsik}. For several further recent results see, e.g., \cite{C5C3, c5-3, ggmv, gmv, GP, gs,gstz}.
Despite these investigations, there are only few cases when $\ex(n,H,F)$ is determined exactly. One difficulty in determining $\ex(n,H,F)$ exactly is that there are few $F$-free graphs that are good candidates for being extremal constructions for maximizing copies of a graph $H$. Our understanding of $F$-free graphs is not deep enough to describe those graphs that are ``largest'' in some sense. An exception is the Tur\'an graph. In this paper we examine when the Tur\'an graph is the extremal construction for these so-called generalized Tur\'an problems.

\begin{defn}
Fix a $k$-chromatic graph $F$ and a graph $H$ that does not contain $F$ as a subgraph\footnote{We include the condition on $H$ to avoid the degenerate case that $\ex(n,H,F)=\mathcal{N}(H,T_{k-1}(n))=0$ which would allow that $H$ is $F$-Tur\'an-good.}. We say that $H$ is $F$-{\it Tur\'an-good} if $\ex(n,H,F)=\mathcal{N}(H,T_{k-1}(n))$ for every $n$ large enough. If $F=K_k$, we use the briefer term $k$-{\it Tur\'an-good}.
\end{defn}

Using this notation, Tur\'an's theorem states that $K_2$ is $k$-Tur\'an-good for every $k>2$, Theorem~\ref{crit-thm} states that $K_2$ is $F$-Tur\'an-good for any $F$ with a color-critical edge and Theorem~\ref{zykov-theorem} states that $K_r$ is $k$-Tur\'an-good for $r<k$. 

Gy\H ori, Pach and Simonovits \cite{gypl} considered the problem to determine which graphs $H$ are $k$-Tur\'an-good. In particular, they showed that a bipartite graph $H$ on $m\geq 3$ vertices containing $\lfloor m/2 \rfloor$ independent edges is $3$-Tur\'an-good. This implies that
the path $P_l$, the even cycle $C_{2l}$ and the Tur\'an graph $T_2(m)$ are all $3$-Tur\'an-good. They also gave the following general theorem.

\begin{theorem}[Gy\H ori, Pach and Simonovits \cite{gypl}]\label{gpl} 
Let $r\ge 3$ and let $H$ be a $(k-1)$-partite graph with $m>k-1$ vertices
containing $\lfloor m/(k-1)\rfloor$ vertex disjoint copies of $K_{k-1}$. Suppose
further that for any two vertices $u$ and $v$ in the same component of $H$, there is a sequence $A_1,\dots,A_s$ of $(k-1)$-cliques in $H$ such that $u\in A_1$, $v\in A_s$, and for any $i<s$, $A_i$ and $A_{i+1}$ share $k-2$ vertices.
Then $H$ is $k$-Tur\'an-good.
\end{theorem}

They mention that this theorem implies that $T_{k-1}(m)$ is $k$-Tur\'an-good and again that paths and even cycles are $3$-Tur\'an-good.
When $H$ is a complete multipartite graph they gave the following theorem.

\begin{theorem}[Gy\H ori, Pach and Simonovits \cite{gypl}]\label{gpl2}
Let $H$ be a complete $r$-partite graph and let $n$ be large enough. If $G$ is an $n$-vertex $K_k$-free graph with the maximum number of copies of $H$, then $G$ is a complete $(k-1)$-partite graph.
\end{theorem}

In \cite{gypl}, the authors remark that a graph $G$ in the above theorem need not be a Tur\'an graph and give an example where the ratio of the sizes of the largest and smallest class of $G$ is not even bounded. 
They also gave an optimization argument to show that $C_4$ is $k$-Tur\'an-good for any $k$. They state the same for $K_{2,3}$, but omit the details. 


Less is known in the case when the forbidden graph $F$ is not a clique. 
Ma and Qiu \cite{mq} proved that a $(k-1)$-partite graph $H$ is $k$-Tur\'an-good if it has 
$k-2$ classes of size $s$ and one class of size $t$ with $s\le t<s+1/2+\sqrt{2s+1/4}$. They also proved a common generalization of Theorems~\ref{crit-thm} and \ref{zykov-theorem}.

\begin{theorem}[Ma and Qiu \cite{mq}]\label{MQ}
Let $F$ be graph with $\chi(F) =k >r$ and a color-critical edge. Then $K_r$ is $F$-Tur\'an-good. Moreover, for $n$ large enough, $T_{k-1}(n)$ is the unique $n$-vertex $F$-free graph with the maximum number of copies of $H$.
\end{theorem}

Let us discuss some simple conditions that force $H$ to not be $F$-Tur\'an-good.
When $\chi(H) \geq \chi(F) =k$, then the Tur\'an graph $T_{k-1}(n)$ contains no copies of $H$, so $H$ cannot be $F$-Tur\'an-good. Observe that if the sizes of color classes of $H$ are very unbalanced, then it is possible that among the complete $(\chi(F)-1)$-partite graphs on $n$ vertices, the Tur\'an graph does not have the maximum number of copies of $H$. When $H$ is a complete multipartite graph, a straightforward optimization can determine which complete $(\chi(F)-1)$-partite graphs contain the maximum number of copies of $H$. Some calculations of this type were performed in \cite{gypl,mq,cnr}. If $F$ is a $k$-chromatic graph with no color-critical edge, then we can add an edge $e$ to the Tur\'an graph $T_{k-1}(n)$ and still have no copy of $F$. If $\chi(H)\leq \chi(F)-2$, then it is easy to see that in the resulting graph (for $n$ large enough) there are copies of $H$ that contain $e$. Therefore, in this case, $H$ is not $F$-Tur\'an-good. Thus, when $F$ has no color-critical edge we can restrict our attention to the case when $\chi(H)=\chi(F)-1$.

\smallskip

The rest of this paper is organized as follows. In Section \ref{cliq} we consider $k$-Tur\'an-good graphs and prove a theorem that is of a similar flavor to Theorem \ref{gpl}. We also show that for any complete multipartite graph $H$, there is a $k_0$ large enough such that if $k \geq k_0$, then $H$ is $k$-Tur\'an-good. In Section~\ref{tura} we consider the case when $F$ is not a clique. Among others, we prove that $P_3$ is $F$-Tur\'an-good for $F$ with $\chi(F) \geq 4$, that $P_4$ and $C_4$ are $C_5$-Tur\'an-good, and that $C_4$ is $F_2$-Tur\'an-good where $F_2$ is the graph of two triangles sharing exactly one vertex. We finish the paper with some concluding remarks and conjectures in Section \ref{rema}. 

\section{Forbidding cliques}\label{cliq}

The main theorem of this section describes a method to construct $k$-Tur\'an-good graphs.

\begin{theorem}\label{turgood}
	Let $H$ be a $k$-Tur\'an-good graph. Let $H'$ be any graph constructed from $H$ in the following way.
	Choose a complete subgraph of $H$ with vertex set $X$, add a vertex-disjoint copy of $K_{k-1}$ to $H$ and join the vertices in $X$ to the vertices of $K_{k-1}$ by edges arbitrarily.
	Then $H'$ is $k$-Tur\'an-good.
\end{theorem}

\begin{proof} By Theorem~\ref{zykov-theorem}, the maximum number of copies of $K_{k-1}$ in a $K_k$-free graph is achieved by the Tur\'an graph $T_{k-1}(n)$. Since $H$ is $k$-Tur\'an-good, the Tur\'an graph $T_{k-1}(n-k+1)$ has the maximum number of copies of $H$ among $K_k$-free graphs on $n-k+1$ vertices. We will show that $T_{k-1}(n)$ has the maximum number of copies of $H'$.
	
	Let $G$ be a $K_{k}$-free graph on $n$ vertices with the maximum number of copies of $H'$. 
	Since $H'$ contains a copy of $K_{k-1}$, the graph $G$ must contain a copy of $K_{k-1}$.
	Let $K$ be a copy of $K_{k-1}$ in $G$. 	
	Every other vertex of $G$ is adjacent to at most $k-2$ vertices of $K$. Let $Y$ be a complete graph that is disjoint from $K$. 
	
	Consider an auxiliary bipartite graph with classes formed by the vertices of $Y$ and $K$ and join two vertices by an edge if they are non-adjacent in $G$.
	
	Suppose this bipartite graph does not have a matching saturating the class $Y$, i.e., a matching that uses every vertex of $Y$. Then, by Hall's theorem, there exists a non-empty subset $Y'$ of $Y$ whose neighborhood in $K$ has size less than $|Y'|$. In the original graph $G$ this means that all of the vertices in $Y'$ are connected to a fixed set of more than $|K|-|Y'|$ vertices in $K$. As $Y'$ and $K$ are complete graphs, this gives a copy of $K_k$ in $G$, a contradiction.
	Therefore, this auxiliary bipartite graph has a matching saturating $Y$ which implies that in $G$ the edges between $Y$ and $K$ are a subgraph of a complete bipartite graph minus a matching saturating $Y$.
	
	On the other hand, in a $(k-1)$-partite Tur\'an graph the edges between $K_{k-1}$ and a clique of size $|Y|$
	form a complete bipartite graph minus a matching saturating the clique of size $|Y|$.
	This implies that there are at least as many ways to join the vertices of a copy of $H$ with a copy of $K_{k-1}$ in a Tur\'an graph as in $G$.
	
	The number of copies of $H'$ is the product of the number of copies of $K_{k-1}$, the number of copies of $H$ on the remaining $n-k+1$ vertices and the number of ways to join the vertices of $K_{k-1}$ and $H$ all divided by the number of times a copy of $H'$ was counted. The first three quantities are maximized by the Tur\'an graph, while the last quantity depends only on $H'$. This implies that the number of copies of $H'$ is maximized by $T_{k-1}(n)$.
\end{proof}

We remark that Theorem \ref{turgood} implies the same results mentioned after Theorem \ref{gpl}. However, neither Theorem \ref{turgood} nor Theorem \ref{gpl} imply the other. They both use copies of $K_{k-1}$ as building blocks and connect them with additional edges, but Theorem \ref{gpl} allows adding many edges. For example, when $k=3$ the only assumptions on $H$ are that $H$ is bipartite and has a matching of size $\lfloor |V(H)|/2\rfloor$. In Theorem \ref{turgood}, when $k=3$, if we build $H$ starting from a single edge, there is always an edge (the one added last) such that its vertices are incident to at most two other vertices. On the other hand, in Theorem \ref{turgood} we do not need a sequence of $(k-1)$-cliques connecting any two vertices. For example we can take two copies of $K_{k-1}$ and connect them with a single edge. The resulting graph is $k$-Tur\'an-good because of Theorem \ref{turgood}.

Nonetheless, both Theorem \ref{gpl} and Theorem \ref{turgood} require copies of $K_{k-1}$ as building blocks. For example, we know that $P_l$ is $3$-Tur\'an-good, and Theorem \ref{gpl2} implies that $P_3$ is $k$-Tur\'an-good, but for longer paths neither Theorem~\ref{gpl} nor Theorem~\ref{turgood} can be applied. We conjecture that paths are $k$-Tur\'an good (see Conjecture~\ref{paths-are-good} in Section~\ref{rema}). Here we are able to show that the maximum number of copies of $P_l$ in $K_k$-free graphs is asymptotic to the number of copies in the Tur\'an graph. In fact, we can replace $K_k$ with any $k$-chromatic graph $F$.

\begin{proposition}\label{path-asym}
   If $F$ is $k$-chromatic with $k>2$, then $\ex(n,P_l,F)=(1+o(1))\cN(P_l,T_{k-1}(n))$.
\end{proposition}

\begin{proof}
We will use spectral methods as they were used in \cite{gstz} and \cite{ggmv}. We use the following simple facts: every path is a walk and a path with more than 2 vertices corresponds to two walks (one starting from each end-vertex of the path). Therefore, the number of walks of length $l-1$ (i.e. having $l-1$ edges) is at least twice the number of paths of length $l-1$. 

For a matrix $M$ let $\mu(M)$ denote the largest eigenvalue of $M$. Now let $A(G)$ be the adjacency matrix of $G$. The number of walks of length $l-1$ in $G$ is at most $\mu(A(G))^{l-1}/n$. (This is well-known, see \cite{gstz} and \cite{ggmv} for simple proofs.)

The largest eigenvalue of the adjacency matrix of graph is well-studied. Babai and Guidulli \cite{guid} and independently Nikiforov \cite{nik} proved that if $F$ has chromatic number $k$ and $G$ is an $n$-vertex $F$-free graph, then $\mu(A(F))=(1-\frac{1}{k-1}+o(1))n$. Therefore, we obtain that 
\[
\ex(n,P_l,F)\le \frac{1}{2}\left(1-\frac{1}{k-1}+o(1)\right)^{l-1}n^{l}.
\]
Now let us count the number of copies of $P_l$ in the Tur\'an graph $T_{k-1}(n)$.
Counting greedily we have $n$ choices for the first vertex. Each subsequent vertex must be in a different class of $T_{k-1}(n)$ from its predecessor and must be different from the previous vertices. Therefore, at each step (after the first) the number of choices for a vertex is at least
\[
n-\left\lceil \frac{n}{k-1}\right\rceil-l+1=\left(1-\frac{1}{k-1}-o(1)\right)n.
\]
In this way, every path is counted twice. Therefore, the number of paths in $T_{k-1}(n)$ is 
\[
\mathcal{N}(P_l,T_{k-1}(n)) = \frac{1}{2}\left(1-\frac{1}{k-1}-o(1)\right)^{l-1}n^{l}.
\]
\end{proof}

We now turn our attention to the case when $H$ is a complete multipartite graph. We begin with a lemma.

\begin{lemma}\label{larg}
   For any graph $H$ there are integers $k_0$ and $n_0$ such that if $k\ge k_0$ and $n\ge n_0$, then for any complete $(k-1)$-partite $n$-vertex graph $G$ we have 
   $\cN(H,G)\le \cN(H,T_{k-1}(n))$.
\end{lemma}

\begin{proof} Suppose $G$ contains the maximum number of copies of $H$ among all complete $(k-1)$-partite graphs on $n$ vertices. Suppose, for the sake of a contradiction, that $G$ is not the Tur\'an graph $T_{k-1}(n)$.

Observe first that we can assume $H$ is a complete multipartite graph. Indeed, if $H$ has chromatic number $r$, then there is a constant number of ways to add edges to $H$ to create a complete $r$-partite graph with $|V(H)|$ vertices. Each copy of $H$ in $G$ is contained in such a complete $r$-partite graph in $G$. Given such a complete $r$-partite graph, we can count the number of copies of $H$ it contains. Therefore, if the number of copies of each such complete $r$-partite graph is maximized by the Tur\'an graph $T_{k-1}(n)$, then the same holds for $H$.

We distinguish two cases.

\smallskip
{\bf Case 1:} There are two vertex partition classes $A$ and $B$ of $G$ such that $|A|\ge |V(H)||B|$.

\smallskip
In this case we will move a vertex from $A$ to $B$ to create a new complete $(k-1)$-partite graph. We will show that this new graph contains more copies of $H$ than $G$. Observe that $H$ intersects $A$ and $B$ in a bipartite graph $H'$. The number of ways to extend $H'$ to $H$ using other classes of $G$ does not change when moving a vertex from $A$ to $B$. Therefore, if the number of copies of each possible $H'$ does not decrease by this change, then the number of copies of $H$ does not decrease either. Moreover, if the number of copies of some $H'$ increases, then the number of copies of $H$ increases, which is a contradiction.

To show that the number of copies of $H'$ increases, assume first that $H'$ is connected. As $H$ is complete multipartite, this implies that $H'$ is a complete bipartite graph $K_{a,b}$ for some $a,b$ with $a+b\le |V(H)|$. We have $\binom{|A|}{a}\binom{|B|}{b}+\binom{|A|}{b}\binom{|B|}{a}$ copies of $H'$ between $A$ and $B$. It is easy to see that this number increases when we move a vertex from $A$ to $B$.

If $H'$ is disconnected, there may be multiple ways to embed it to the classes $A$ and $B$. However, for every such embedding with $a'$ and $b'$ vertices in $A$ and $B$, the same argument as above shows that the number of such embeddings increases when we move a vertex from $A$ to $B$, thus the number of copies of $H'$ increases.

\smallskip

{\bf Case 2:} For every pair of partition classes $A$ and $B$ in $G$, we have  $|A|< |V(H)||B|$.

\smallskip

Let us fix $\epsilon>0$ and choose $k_0$ such that $k_0-1>|V(H)|/\epsilon$. Now assume that $k \geq k_0$. Then the average size of the classes in $G$ is 
\[
\frac{n}{k-1}\leq \frac{n}{k_0-1}<\frac{\epsilon n}{|V(H)|}.
\] 
Therefore, the size of each class $X$ of $G$ satisfies
\[
\frac{1}{(k-1)|V(H)|} n \leq |X| \leq \frac{|V(H)|}{k-1} n < \epsilon n.
\]
The graph $G$ is not a Tur\'an graph, so it has classes $A$ and $B$ such that $|A|>|B|+1$. 
Let us move a vertex from $A$ to $B$ to create a new complete $(k-1)$-partite graph $G'$. Let us count the number of copies of $H$ destroyed and created when moving a vertex from $A$ to $B$.  It is well-known and easy to see that $G'$ has more edges than $G$.

Those copies of $H$ in $G$ that do not have any edge from $A$ to $B$ remain in the graph.
For each edge $uv$ between $A$ and $B$ consider the copies of $H$ where $u$ and $v$ are the only vertices of $H$ in $A \cup B$. Observe that their number does not depend which vertices $u$ and $v$ we choose from $A$ and $B$. We can greedily pick a vertex from each of the other $|V(H)|-2$ distinct classes to extend $uv$ to a unique such copy of $H$. At each step we can choose from at least $n-|V(H)|\epsilon n$ vertices. Therefore, the number of such copies of $H$ is at least
\[
((1-|V(H)| \epsilon)n)^{|V(H)|-2}.
\]
As there are more edges between $A$ and $B$ in $G'$ than in $G$, we have created at least
$((1-|V(H)| \epsilon)n)^{|V(H)|-2}$ new copies of $H$.

Now consider a copy of $H$ that has at least two edges between $A$ and $B$.
Such copy of $H$ has $p\geq 3$ vertices in $A\cup B$. These $p$ vertices can be extended to a copy of $H$ in at most $n^{|V(H)|-p}$ ways. Now pick an arbitrary bipartite subgraph $H'$ of $H$ with $p\ge 3$ vertices. We claim that the number of copies of $H'$ in $A\cup B$ decreases by at most $\epsilon c n^{p-2}$ when we move a vertex $v$ from $A$ to $B$ for some constant $c$ that depends only on $H$. 

Indeed, consider a proper $2$-coloring of $H'$ with $a$ vertices of color red and $b$ vertices of color blue. We may suppose that $v$ is in our copy of $H'$ otherwise $H$ is unchanged.
Then we have to pick $a-1$ vertices from $A$ and $b$ vertices from $B$ (or vice versa) to form a copy of $H'$. Therefore, we start with $\binom{|A|}{a-1}\binom{|B|}{b}+\binom{|A|}{b}\binom{|B|}{a-1}$ copies of $H'$ and, after moving $v$, we end up with $\binom{|A|-1}{a-1}\binom{|B|+1}{b}+\binom{|A|-1}{b}\binom{|B|+1}{a-1}$ copies of $H'$. Recall that only $|A|$ and $|B|$ depend on $n$ and $\epsilon$. Simple expansion gives
\[
\binom{|A|-1}{a-1}\binom{|B|+1}{b}=\frac{(|A|-a+1)(|B|+1)}{|A|(|B|+1-b)}\binom{|A|}{a-1}\binom{|B|}{b}.
\]
Therefore, the difference between the first terms of these counts of $H'$ is
\begin{align*}
& \frac{(|A|-a+1)(|B|+1)-|A|(|B|+1-b)}{|A|(|B|+1-b)}\binom{|A|}{a-1}\binom{|B|}{b}
\\
&\le \frac{|A|b-|B|(a-1)}{|A|(|B|+1-b)} |A|^{a-1}|B|^b \le 
\frac{|A|b-|B|(a-1)}{|B|+1-b}(\epsilon n)^{a+b-2}=c_0 (\epsilon n)^{a+b-2},\end{align*}
where $c_0\le |V(H)|^2(b+1)$. A similar bound can be obtained for the difference of the second terms, proving our claim.


This shows that the number of copies of $H$ that have more than one edge between $A$ and $B$ decreases by at most $c\epsilon n^{|V(H)|-2}$, thus the total number of copies of $H$ increases, a contradiction.
\end{proof}

When $H$ is a star $S_t$, Cutler, Nir and Radcliffe \cite{cnr} state that numerical evidence suggests for small $t$ (i.e., large $k$) that $S_t$ is $k$-Tur\'an-good. We can confirm this statement for every multipartite $H$. Indeed, Theorem~\ref{gpl2} and Lemma~\ref{larg} together imply the following theorem.

\begin{theorem}\label{multi}
For every complete multipartite graph $H$ there is an integer $k_0$ such that if $k\ge k_0$, then $H$ is $k$-Tur\'an-good.
\end{theorem}

We believe that Theorem~\ref{multi} should hold for any graph $H$. See Conjecture~\ref{conj3} in Section~\ref{rema} for details.

\section{Forbidding non-cliques}\label{tura}

We begin this section with a simple proposition.

\begin{prop}\label{p3-critial}
If $F$ is a graph with chromatic number $\chi(F)=k \geq 4$ and a color-critical edge, then $P_3$ is $F$-Tur\'an-good.
\end{prop}

\begin{proof} 
	Fix an $F$-free $n$-vertex graph $G$ and let $p(G)$ be the number of induced copies of $P_3$. 
	Let us count the number of pairs $(uv,w)$ where $uv$ is an edge in $G$ and $w$ is a vertex in $G$ that is distinct from $u$ and $v$. Clearly, there are $|E(G)|(n-2)$ such pairs. On the other hand, on any three vertices there is at most one triangle or one induced $P_3$. Moreover, each triangle consists of three such pairs $(uv,w)$ and every induced $P_3$ consists of two such pairs $(uv,w)$. Thus
	\begin{equation}\label{pair-count}
	2p(G)+3\mathcal{N}(K_3,G) \leq |E(G)|(n-2).
	\end{equation}
	For the graph $G=T_{k-1}(n)$ we have equality in (\ref{pair-count}). 
	By Theorems \ref{crit-thm} and \ref{MQ}, we have that $|E(G)| \leq |E(T_{k-1}(n))|$ and
	$\mathcal{N}(K_3,G) \leq \mathcal{N}(K_3,T_{k-1}(n))$.
	Counting  copies of $P_3$ in $G$ gives
	\begin{align*}
	\ex(n,P_3,F) = \mathcal{N}(P_3,G)&=p(G)+3 \mathcal{N}(K_3,G)=(p(G)+\frac{3}{2}\mathcal{N}(K_3,G))+\frac{3}{2}\mathcal{N}(K_3,G) \\
	&\le (p(G)+\frac{3}{2}\mathcal{N}(K_3,G))+\frac{3}{2}\mathcal{N}(K_3,T_{k-1}(n))\\
	&\le \frac{1}{2}|E(G)|(n-2)+\frac{3}{2}\mathcal{N}(K_3,T_{k-1}(n))\\
	&\le \frac{1}{2}|E(T_{k-1}(n))|(n-2)+\frac{3}{2}\mathcal{N}(K_3,T_{k-1}(n))=\mathcal{N}(P_3,T_{k-1}(n)).
	\end{align*}
\end{proof}

We believe that
the condition on the chromatic number of $F$ can be reduced to $3$ in Proposition~\ref{p3-critial}.

\subsection{Forbidding a book}

Recall that a {\it book} $B_k$ is the graph of $k$ triangles all sharing exactly one common edge. Note that book $B_2$ is simply the graph resulting from removing an edge from $K_4$. We will show that both $C_4$ and $P_4$ are $B_2$-Tur\'an-good.

Let $\overline{M_k}$ be the complement of the graph of $k$ independent edges, i.e., $\overline{M_k}$ is a clique on $2k$ vertices with the edges of a perfect matching removed.
Let $\overline{M_k}^+$ be the graph resulting from adding an edge to $\overline{M_k}$, i.e, 
$\overline{M_k}^+$ is the graph of a clique on $2k$ vertices with all but one of the edges of a perfect matching removed. Note that $\overline{M_k}$ and $\overline{M_k}^+$ differ by a single edge and that $\chi(\overline{M_k})=k$ and $\chi(\overline{M_k}^+)=k+1$, i.e., the graph $\overline{M_k}^+$ has a color-critical edge.
Also note that when $k=2$, we have that $\overline{M_k}$ is the cycle $C_4$ and $\overline{M_k}^+$ is the book $B_2$  (i.e., $K_4$ minus an edge).

\begin{lemma}\label{matching-complement}
Let $H$ be a $2k$-vertex graph consisting of two vertex-disjoint copies of $K_k$ joined together with edges arbitrarily. If $\overline{M_k}$ has the maximum number of copies of $H$ among all $2k$-vertex $\overline{M_k}^+$-free graphs, then $H$ is $\overline{M_k}^+$-Tur\'an-good. In particular, $\overline{M_k}$ is $\overline{M_k}^+$-Tur\'an-good.
\end{lemma}

\begin{proof} We can count the copies of $\overline{M_k}$ by counting the number of ways to choose a pair of disjoint copies of $K_k$ and then counting the number of copies of $\overline{M_k}$ spanned by these two copies of $K_k$. We show that each of these quantities is maximized among $n$-vertex $\overline{M_k}^+$-free graphs by the Tur\'an graph $T_{k}(n)$.

By Theorem~\ref{MQ}, for $n$ large enough, the Tur\'an graphs $T_k(n)$ and $T_k(n-k)$ contain the maximum number of copies of $K_k$ among all $n$-vertex and $(n-k)$-vertex $\overline{M_k}^+$-free graphs. Therefore, $T_k(n)$ maximizes the number of pairs of disjoint copies of $K_k$.
 In an $\overline{M_k}^+$-free graph, if two disjoint copies of $K_k$ span a copy of $\overline{M_k}$, then there can be no further edges between them as otherwise we have a copy of $\overline{M_k}^+$. Thus, any two disjoint copies of $K_k$ span at most one copy of $\overline{M_k}$. Observe that in $T_k(n)$ any pair of disjoint copies of $K_k$ span exactly one copy of $\overline{M_k}$. As the number of copies of $H$ on $2k$-vertices is maximized by $\overline{M_k}$ we have that
 the number of copies of $H$ is maximized in $T_k(n)$.
\end{proof}

When $k=4$ the graphs $P_4$ and $C_4$ are both candidates for the graph $H$ in Lemma~\ref{matching-complement}. This gives the following proposition.

\begin{proposition}\label{neew}
 The cycle $C_4$ and path $P_4$ are $B_2$-Tur\'an-good.
\end{proposition}

We remark that one can also obtain that $C_4$ is $B_2$-Tur\'an-good from a result of Pippenger and Golumbic \cite{pg}. They showed that $T_2(n)$ contains the largest number of \emph{induced} copies of $C_4$ among $n$-vertex graphs. As every copy of a $C_4$ is induced in a $B_2$-free graph, this implies that $C_4$ is $B_2$-Tur\'an-good.

\subsection{Forbidding odd cycles}
 Gerbner, Gy\H ori, Methuku and Vizer \cite{ggmv} counted paths and even cycles when forbidding an odd cycle. In particular, they
 proved that for
 any $k \geq 1$ and $l \geq 2$,
 \begin{align*}
 &\ex(n,P_{l},C_{2k+1}) =  (1+o(1))\left(\frac{n}{2}\right)^l = (1+o(1)) \mathcal{N}(P_l,T_2(n)) \\
 &\ex(n,C_{2l},C_{2k+1}) =  (1+o(1))\frac{1}{2l} \left(\frac{n^2}{4}\right)^l= (1+o(1)) \mathcal{N}(C_{2l},T_2(n)).
 \end{align*}

In this subsection we show that both $P_4$ and $C_4$ are $C_5$-Tur\'an-good, i.e, the results above are exact for $k=2$ and $P_4$ and $C_4$, respectively. In the case of $P_4$ we also obtain a stability result.

\begin{theorem}\label{p4-c5}
The path $P_4$ is $C_5$-Tur\'an-good. Moreover, if $G$ is a $C_5$-free graph on $n$ vertices and $G$ has $\alpha$ edges contained in triangles, then the number of copies of $P_4$ in $G$ is at most
\[
\cN(P_4,T_2(n))-(1+o(1))\alpha \frac{n^2}{12}.
\]

\end{theorem}

\begin{proof}
Let $G$ be an $n$-vertex $C_5$-free graph.
We will show that every edge of $G$ is contained in at most $3\lfloor n/2-1\rfloor \lceil n/2-1 \rceil$ copies of $P_4$. As the number of edges is maximized in the Tur\'an graph $T_2(n)$ and in the Tur\'an graph every edge is contained in $3\lfloor n/2-1\rfloor \lceil n/2-1 \rceil$ copies of $P_4$, this will show that $P_4$ is $C_5$-Tur\'an-good.
In order to prove the second part of the theorem we will examine the number of copies of $P_4$ containing a fixed edge $e$ where $e$ is contained in a triangle in $G$.

Consider an edge $uv$ and let $G'$ be obtained from $G$ by deleting $u$ and $v$. 
As $G'$ is a $C_5$-free graph on $n-2$ vertices and $n-2$ is large enough, Theorem~\ref{crit-thm} implies that $G'$ satisfies 
\[
|E(G')| \leq |E(T_2(n-2))| = \lfloor n/2-1\rfloor \lceil n/2-1\rceil.
\]
Now partition $V(G')$ into a set $A$ of vertices adjacent to both $u$ and $v$, a set $B$ of vertices adjacent to $u$ but not $v$, a set $C$ of vertices adjacent to $v$ but not $u$, and a set $D$ of the remaining vertices (not adjacent to $u$ nor $v$). As $G$ is $C_5$-free, no vertex in $V(G')$ is adjacent to a vertex in $A\cup B$ and a distinct vertex in $A\cup C$. 

Observe that if $|A|\ge 1$, then there is no edge between $B$ and $C$.
If $|A|\ge 2$, then there is no edge between $A$ and $B\cup C$.
If $|A|\ge 3$, then there is no edge in $A$.

Consider two vertices $x$ and $y$ of $G'$ (necessarily distinct from $u$ and $v$). 
Let $f(x,y)$ denote the number of copies of $P_4$ in $G$ containing the edge $uv$ and vertices $x$ and $y$. If $x,y\in A$ and $xy$ is an edge, then $f(x,y)=6$
and if $xy$ is not an edge, then $f(x,y)=2$. If $x\in A$ and $y\in B\cup C$ and $xy$ is an edge, then $f(x,y)=4$ and if $xy$ is not an edge, then $f(x,y)=1$. If $x\in B$ and $y\in C$ and $xy$ is an edge, then $f(x,y)=3$ and if $xy$ is not an edge, then $f(x,y)=1$. If $x,y\in B$ and $xy$ is an edge, then $f(x,y)=2$ and if $xy$ is not an edge, then $f(x,y)=0$. The same argument holds when $x,y \in C$. If $x\in D$ is not adjacent to $y$, then $f(x,y)=0$. If $x$ is adjacent to $y$ and $y\in A$, then $f(x,y)=2$ and if $y\in B\cup C$ then $f(x,y)=1$ and if $y\in D$, then $f(x,y)=0$.
Let 
\[
q(u,v):=\sum_{x,y\in V(G')} f(x,y),
\]
i.e., $q(u,v)$ is the number of copies of $P_4$ containing the edge $uv$.
We determine an upper-bound on $q(u,v)$ in four cases based on the size of $A$. 

\smallskip

\textbf{Case 1:} $A=\emptyset$, i.e., $uv$ is not contained in any triangles. 

\smallskip

For every pair $x,y$ of vertices, $f(x,y)$ depends on which of the three sets $B$, $C$ and $D$ they belong to and whether $x$ and $y$ are adjacent in $G'$. Observe that when $x \in B$ and $y \in C$, then $f(x,y) \leq 3$ if $xy$ is an edge and $f(x,y) \leq 1$ otherwise. In all other cases $f(x,y) \leq 2$ if $xy$ is an edge and $f(x,y) \leq 0$ otherwise. 
Therefore, $q(u,v)\le 2|E(G')|+|C||B|$. Both terms of this sum are maximized by the Tur\'an graph, $T_2(n-2)$, so $q(u,v) \leq 3\lfloor n/2-1\rfloor \lceil n/2-1 \rceil$.

\smallskip

\textbf{Case 2:} $A=\{w\}$. 

\smallskip

Then $\sum_{y\in V(G')}f(w,y)\le 4(n-3)$. If $x\neq w\neq y$, then $f(x,y)\le 2$. Indeed, if $x$ and $y$ are adjacent, then it is impossible that one of them is in $B$ and the other is in $C$. Moreover, if $x$ and $y$ are not adjacent, then $f(x,y)=0$. Therefore, we have $q(u,v)\le 4(n-3)+2|E(G')|\le (1+o(1))n^2/2$.

\smallskip

\textbf{Case 3:} $A=\{w,w'\}$.

\smallskip

If $\{x,y\}=\{w,w'\}$, then $f(x,y)\le 6$. By the same reasoning as in Case 2, if $|\{x,y\}\cap\{w,w'\}|=1$, then $f(x,y)\le 4$, and if $|\{x,y\}\cap\{w,w'\}|=0$, then $f(x,y)\le 2$. Moreover, in this latter case, if $x$ and $y$ are not adjacent, then $f(x,y)=0$. Therefore, we obtain $q(u,v)\le 6+8(n-3)+2|E(G')|\le (1+o(1))n^2/2$. 

\smallskip

\textbf{Case 4:} $|A|=m\ge 3$.

\smallskip

Then we know $f(x,y)\le 2$ if $x,y\in A$, since $xy$ is not an edge of $G'$. Furthermore, 
the only other case when $f(x,y)\ge 2$ is when $x\in A$ and $y\in D$ are adjacent and we have $f(x,y) =2$. Observe that $y\in D$ can be adjacent to at most one $x\in A$. Therefore this situation can occur at most once for each element of $D$, i.e., at most $n-m-2$ total times. 
All other pairs $x,y$ have $f(x,y) \leq 1$ and therefore $q(u,v) \le \binom{n}{2}+\binom{m}{2}+n-m-2$.

\smallskip

Let us call a subgraph of $G$ a \textit{large book} if it consists of the book spanned by the edge $uv$ and all the $m\ge 3$ common neighbors $w_1,\dots, w_m$ of $u$ and $v$. Observe that each edge of the form $uw_i$ or $vw_i$ in $G$ is contained in the single triangle $uvw_i$ as otherwise we can form a $C_5$ with $w_j$ for some $j \neq i$ (as $m \geq 3$), a contradiction.

This implies that large books are pairwise edge-disjoint.
Therefore, we can calculate the sum of $q(u,v)$ for all edges $uv$ contained in a large book by summing them for each large book. In a large book $H$ with $m+2$ vertices, we have $2m$ edges each contained in exactly one triangle (thus we can use Case 2) and one edge contained in exactly $m$ triangles (where we use Case 4). Therefore,
\[
\sum_{uv\in E(H)} q(u,v)\le 2m(1+o(1))n^2/2+\binom{n}{2}+\binom{m}{2}+n-m-2\le (2m+1)(1+o(1))n^2/2 + \binom{m}{2}
\]
This implies that for edges in large books, the average of $q(u,v)$ is  
\[
\frac{1}{2m+1}\sum_{uv\in E(H)} q(u,v)\le (1+o(1))n^2/2+\frac{1}{2m+1}\binom{m}{2} = (1+o(1))n^2/2.
\]

For the at most $\lfloor n^2/4\rfloor-\alpha$ edges not in triangles, we have $q(u,v)\le 3\lfloor n/2-1\rfloor \lceil n/2-1 \rceil$ by Case 1. 
For edges in triangles but not in large books, we have $q(u,v)\le (1+o(1))n^2/2$ by Cases 2 and 3. Therefore, 
\begin{align*}
  3\cN(P_4,G) & =\sum_{uv\in E(G)} q(u,v) \\
  & \le  3\left\lfloor {n}/{2}-1\right\rfloor\lceil {n}/{2}-1\rceil (\lfloor n^2/4\rfloor-\alpha)+(1+o(1))\alpha n^2/2= \\
  &=3\cN(P_4,T_2(n))-(1+o(1))\alpha n^2/4,   
\end{align*}
completing the proof.
\end{proof}


\begin{lemma}\label{even-path-to-even-cycle}
Fix a graph $F$ and
let $G_0$ be a complete bipartite graph with $\mathcal{N}(P_{2k},G_0) = \ex(n,P_{2k},F)$. Then $G_0$ satisfies $\mathcal{N}(C_{2k},G_0) = \ex(n,C_{2k},F)$.
\end{lemma}

\begin{proof}
Let $G$ be an $n$-vertex $F$-free graph with $\ex(n,C_{2k},F)$ copies of $C_{2k}$. Observe that $\cN(P_{2k},G) \leq \ex(n,P_{2k},F)$. 
Every copy of $C_{2k}$ contains $2k$ copies of $P_{2k}$ and each copy of $P_{2k}$ is contained in at most one $C_{2k}$. Thus,
 \[
 2k \cdot \cN(C_{2k},G)\le \mathcal{N}(P_{2k},G).
 \]
 Note that copies of $P_{2k}$ not contained in a $C_{2k}$ are not counted here.
 As $G_0$ is a complete bipartite graph, every copy of $P_{2k}$ in $G_0$ is contained in a $C_{2k}$. Therefore,
 \[
 2k \cdot \cN(C_{2k},G_0)=\cN(P_{2k},G_0).
 \]
 Thus, 
 \begin{align*}
 \ex(n,C_{2k},F)& =\cN(C_{2k},G)\le \cN(P_{2k},G)/2k\le \ex(n,P_{2k},F)/2k \\ & = \cN(P_{2k},G_0)/2k=\cN(C_{2k},G_0).
 \end{align*}
\end{proof}

Theorem~\ref{p4-c5} and Lemma~\ref{even-path-to-even-cycle} imply the following corollary.

\begin{cor}\label{c4-c5}
    Let $F$ be a 3-chromatic graph. If $P_{2k}$ is $F$-Tur\'an-good, then $C_{2k}$ is $F$-Tur\'an-good. In particular, $C_4$ is $C_5$-Tur\'an-good.
\end{cor}

\subsection{Forbidding fans}

Until this point we only considered forbidden graphs that have a color-critical edge. By Theorem~\ref{crit-thm}, an extremal graph for a $k$-chromatic graph $F$ without a color-critical edge has more edges than $T_{k-1}(n)$. This suggests that there may not be graphs $H$ that are $F$-Tur\'an-good in this case. Proposition~\ref{noncrit} below shows that this is false in general.

For $k \geq 2$, the {\it $k$-fan} $F_k$ is the graph of $k$ triangles all sharing exactly one common vertex. Note that the fan $F_k$ does not contain a color-critical edge.
Erd\H os, F\"uredi, Gould and Gunderson \cite{efgg} determined the exact Tur\'an number of $F_2$ for $n$ large enough. 

\begin{theorem}[Erd\H os, F\"uredi, Gould and Gunderson \cite{efgg}]\label{fan-turan}
Let $F_2$ be the graph of two triangles sharing exactly one vertex. Then, for $n$ large enough, the unique extremal graph for $F_2$ is the graph resulting from adding a single edge to one class of the Tur\'an graph $T_2(n)$. Thus, for $n$ large enough,
\[
\ex(n,F_2)=\left\lfloor \frac{n^2}{4}\right\rfloor+1.
\]
\end{theorem}

We use Theorem~\ref{fan-turan} to show that $C_4$ is $F_2$-Tur\'an-good.

\begin{proposition}\label{noncrit}
 The cycle $C_4$ is $F_2$-Tur\'an-good.
\end{proposition}

\begin{proof} Let $n$ be large enough and $G$ be an $n$-vertex $F_2$-free graph. If $G$ has more than $\lfloor n^2/4\rfloor$ edges, then Theorem~\ref{fan-turan} gives the exact structure of $G$. In particular, $G$ has $\cN(C_4,T_2(n))$ copies of $C_4$ (observe that the edge added to the $T_2(n)$ is not in any copy of $C_4$).

Therefore, we may assume that $G$ has at most $\lfloor n^2/4\rfloor$ edges. Let $uv$ be an arbitrary edge of $G$. We claim that $uv$ is in at most $\lfloor (n-2)^2/4\rfloor$ copies of $C_4$, which will complete the proof as it implies $\cN(C_4,G)\le \frac{1}{4} \lfloor n^2/4\rfloor\lfloor (n-2)^2/4\rfloor=\cN(C_4,T_2(n))$.
We distinguish two cases.

\smallskip

{\bf Case 1:} $uv$ is not contained in a $K_4$.

\smallskip

Delete $u$ and $v$ from $G$ and let $G'$ be the resulting graph. Every edge of $G'$ forms at most one $C_4$ with $uv$ and we count each such $C_4$ exactly once this way. 
The number of edges in $G'$ is at most $\lfloor (n-2)^2/4\rfloor+1$ by Theorem~\ref{fan-turan}. We are done unless  $|E(G')| = \lfloor (n-2)^2/4\rfloor+1$ and every edge of $G'$ forms a $C_4$ with $uv$ in $G$. 
So, by Theorem~\ref{fan-turan}, we may assume that $G'$ is a $T_2(n-2)$ with classes $A$ and $B$ and an extra edge $xy$ in class $A$.
Without loss of generality, we may assume $xyuvx$ is a $C_4$ in $G$. Thus $vx$ and $uy$ are edges of $G$.

Now let $z$ be an arbitrary vertex of $B$. 
Observe that if $ux$ (or $vy$) is an edge of $G$, then we have a copy of $F_2$ spanned by the two triangles $xyzx$ and $uvxu$ (or $uvyu)$, a contradiction. On the other hand, the edges $xz$ and $yz$ are each in a $C_4$ with $uv$. This implies that $zu$ and $zv$ are both edges of $G$. But then the two triangles $uvzu$ and $xyzx$ span a copy of $F_2$, a contradiction.

\smallskip

{\bf Case 2:} $uv$ is contained in a $K_4$.

\smallskip

Let $u$, $v$, $x$ and $y$ be the vertices of a $K_4$. Delete these four vertices and let $G'$ be the resulting $F_2$-free graph on $n-4$ vertices.
 Observe that each vertex of $G'$ is adjacent to at most one of ${u,v,x,y}$ as otherwise we have an $F_2$ in $G$. Therefore, each edge of $G'$ forms at most one $C_4$ with $uv$.
By Theorem~\ref{fan-turan} we have $|E(G')| \leq \lfloor (n-4)^2/4\rfloor+1$.
Therefore, the number of copies of $C_4$ containing $uv$ is at most $2+\lfloor (n-4)^2/4\rfloor+1\le \lfloor (n-2)^2/4\rfloor$, completing the proof.
\end{proof}

\section{Concluding remarks}\label{rema}

Theorem \ref{turgood} and a weaker version of Proposition~\ref{p3-critial} previously appeared in the authors' first and second arXiv version of \cite{GP}, but were ultimately not included in the published version.

Gy\H ori, Pach and Simonovits \cite{gypl} also studied when the Tur\'an graph is the unique extremal graph. We have so far avoided this for simplicity. Let us say that a graph $H$ is \emph{strictly $F$-Tur\'an-good} for a $k$-chromatic graph $F$ if for every $n$ large enough, $T_{k-1}(n)$ is the unique $F$-free graph with $\ex(n,H,F)$ copies of $H$. By Theorem~\ref{MQ}, if $F$ does not have a color-critical edge, then there is no strictly $F$-Tur\'an-good graph. However, it is not hard to show that our results when $F$ has a color-critical edge hold for the strict version as well.

Let us conclude with several natural conjectures supported by the results in this paper.

\begin{conjecture}\label{paths-are-good}
For every pair of integers $l$ and $k$,
the path $P_l$ is $k$-Tur\'an-good.
\end{conjecture}

Theorem~\ref{gpl} and Corollary~\ref{p3-critial} imply that the conjecture holds for $l=3$ and $k \geq 3$. Proposition~\ref{path-asym} shows that the conjecture holds asymptotically.

\begin{conjecture}\label{conj3}
For every graph $H$ there is an integer $k_0$ such that if $k\ge k_0$, then $H$ is $k$-Tur\'an-good.
\end{conjecture}

Theorem~\ref{multi} implies that the conjecture is true for $H$ a complete multipartite graph. As noted in the introduction, this conjecture cannot in general be extended to hold for small $k$. Observe that Conjecture~\ref{conj3} would imply that if we increase $k$, sooner or later every graph becomes $k$-Tur\'an-good. In some of our examples if a graph was $k$-Tur\'an-good, then it was also $(k+1)$-Tur\'an-good. 
We do not know if this behavior holds for every graph $H$.

\begin{conjecture} The path $P_k$ and the even cycle $C_{2k}$ are $C_{2l+1}$-Tur\'an-good.
\end{conjecture}

 The asymptotic version of this statement was proved in \cite{ggmv}. By Theorem~\ref{gpl}, the conjecture holds when $l = 1$. In this paper, we proved that it holds for $P_3$ when $l \geq 1$ (Corollary~\ref{p3-critial})
 and for $P_4$ and $C_4$ when $l=2$ (Theorem~\ref{p4-c5} and Corollary~\ref{c4-c5}).

\end{document}